\documentclass[reqno]{amsart}%[a4paper,10pt,reqno]{amsart}

\usepackage{mathrsfs, amsfonts, amsmath, amssymb, amscd}
\usepackage{amsbsy,bm, amsthm, enumerate}
\usepackage{textcomp}
\usepackage[english]{babel}
\usepackage{amsmath,amssymb,amsthm}
\usepackage[centering]{geometry}
\usepackage{xcolor}
\usepackage{caption}
\usepackage{booktabs}
\usepackage{graphicx}
\usepackage{subcaption}
\usepackage{amsmath}  % 数学宏包
\usepackage{hyperref} % 开启超链接
\usepackage{hyperref}
\hypersetup{
    hidelinks = true        % 隐藏所有链接的边框和颜色
}
\renewcommand{\epsilon}{\varepsilon}            %%  changes epsilon
\newtheorem{theorem}{Theorem}[section]   %% definition of theorem environment
\newtheorem*{theorem*}{Theorem}          %% a theorem environment without numbering
\newtheorem{lemma}[theorem]{Lemma}

\theoremstyle{definition}

\newtheorem{example}{Example}[section]  %

\numberwithin{equation}{section}

\title[inverse curvature flow]
{On length-preserving and area-preserving inverse curvature flows in the hyperbolic plane}

\author{Zhishuai Liu\and Guoxin Wei}

\address{School of Mathematical Sciences, South China Normal University,
	Guangzhou 510631, People's Republic of China}
\email{liuzs@m.scnu.edu.cn}

\address{School of Mathematical Sciences, South China Normal University,
	Guangzhou 510631, People's Republic of China}
\email{weiguoxin@tsinghua.org.cn}

\makeatletter
\@namedef{subjclassname@2020}{ Mathematics Subject Classification}
\makeatother

\begin{document}
	
	\begin{abstract}
In this paper, we study the area-preserving and length-preserving $\kappa^\alpha$-type curvature flows of smooth, closed, convex curves in the two-dimensional hyperbolic plane $\mathbb H^2$ for $\alpha<0$ and prove that convexity is preserved along the flows. Assuming that the flows exist for all time, we show that the evolving curves converge smoothly to geodesic circles. Furthermore,  we also derive a sufficient condition for global existence of the flows.
\end{abstract}
	
	\maketitle

	%%%%%%%%%%%%%%%%%%%%%%
	\section{Introduction}\label{sect:1}
	%%%%%%%%%%%%%%%%%%%%%%
A one-parameter family of curves $X:S^1\times[0,T)\to M^2(K)$ satisfying
\begin{equation}
\left\{
\begin{aligned}
&\frac{\partial X}{\partial t} = \bigl(\phi_i(t)-\kappa^\alpha\bigr)\nu, \\
&X(u,0)=X_0,
\end{aligned}
\right.\ \ \ \ \   i=1,2,\  \alpha\neq 0,
\label{1.1}
\end{equation}
is called the $\kappa^\alpha$-type curvature flow,
where $M^2(K)$ is the simply connected two-dimensional space form of constant sectional curvature $K\in\{-1,0\}$ (i.e. $M^2(0)$ is Euclidean plane $\mathbb{R}^2$, $M^2(-1)$ is hyperbolic plane $\mathbb{H}^2$), the initial curve $X_0$ is strictly convex, $\alpha$ is a constant, $\kappa$ denotes the curvature of $\gamma_t=X(S^1,t)\subset M^2(K)$, and $\nu$ denotes the unit normal to $\gamma_t$. The term $\phi_i(t)$ is chosen so that the flow preserves either the enclosed area $A(t)$ or the length $L(t)$ of $\gamma_t$.
More precisely, if
\begin{equation}
\phi_1(t)=\frac{1}{L(t)}\int_{\gamma_t}\kappa^\alpha\,ds,
\label{1.2}
\end{equation}
then the flow \eqref{1.1} preserves the area $A(t)$ of the domain $\Omega_t$ enclosed by $\gamma_t$. If
\begin{equation}
\phi_2(t)=\frac{1}{\int_{\gamma_t}\kappa\,ds}\int_{\gamma_t}\kappa^{\alpha+1}\,ds,
\label{1.3}
\end{equation}
then the flow \eqref{1.1} preserves the length $L(t)$ of $\gamma_t$. A smooth closed curve $\gamma_t$ is called convex if its curvature $\kappa$ is positive. By the Gauss--Bonnet theorem applied to the domain enclosed by $\gamma_t$ in $M^2(K)$, one  has $\int_{\gamma_t}\kappa\,ds=2\pi-KA(t)$.
\par
The flow \eqref{1.1} has been studied extensively in the Euclidean plane $\mathbb{R}^2$.
For $\alpha>0$, Tsai and Wang \cite{18} proved that the $\kappa^\alpha$-type flow in $\mathbb{R}^2$ evolves convex curves smoothly to round circles in both the area-preserving and length-preserving cases.
For $\alpha<0$, Gao, Pan and Tsai \cite{7,8} studied the corresponding flow for embedded, convex and smooth initial curves. They showed that the length-preserving flow preserves the length while the enclosed area is non-decreasing, whereas the area-preserving flow preserves the enclosed area while the length is non-increasing.
In contrast to the flow considered in \cite{13,16}, singularities may develop along the flow \eqref{1.1}; in particular, the curvature may blow up, as illustrated by explicit examples in \cite{7,8}.
They further proved that, if the flow exists for all time, or equivalently, if no finite-time curvature blow-up occurs, then the evolving curve converges smoothly to a round circle as time tends to infinity.

In \cite{20}, Wei and Yang studied the flows \eqref{1.1} for $\alpha>0$ in hyperbolic plane $\mathbb{H}^2$, that is, they studied the area-preserving and length-preserving $\kappa^\alpha$-type curve flow in the hyperbolic plane $\mathbb{H}^2$. Using Tso's method, they established a time-independent upper bound for the curvature and proved that the evolving curves converge smoothly and exponentially to a geodesic circle.
For further results in this direction in hyperbolic space, we refer the reader to \cite{2,3,21,22}.

Motivated by the results of Wei and Yang (\cite{20}) on the flow \eqref{1.1} for $\alpha>0$ in the hyperbolic plane $\mathbb{H}^2$, in this paper,  we study the flow \eqref{1.1} for $\alpha<0$ in $\mathbb{H}^2$. Our main result is as follows.

\begin{theorem}
Assume that $\alpha<0$ and that $\gamma_0(S^1)$ is a smooth, closed and convex curve. Consider the flow \eqref{1.1}, and assume that the curvature $\kappa$ does not blow up in finite time during the evolution. Then the flow \eqref{1.1}, with $\phi_1(t)$ defined by \eqref{1.2} $($resp. $\phi_2(t)$ by \eqref{1.3}$)$, admits a unique smooth solution $\gamma_t$ for  $t\in [0,\infty)$. Moreover, $\gamma_t$ remains convex and converges smoothly and exponentially as $t\to\infty$ to a geodesic circle enclosing the same area as $\gamma_0$ (resp. having the same length as $\gamma_0$).
\end{theorem}

The paper is organized as follows. In Section 2, we collect some preliminary facts on the geometry of curves in the hyperbolic plane and derive the evolution equations associated with the flow \eqref{1.1}. In Section 3, we prove that convexity is preserved along the flow. Under the assumption that the curvature does not blow up in finite time, we then derive uniform geometric estimates for the evolving domains and show that the flow exists for all time and converges smoothly and exponentially to a geodesic circle. In Section 4, we establish a sufficient condition for global existence. In Section 5, we present an example for which the curvature remains bounded for all time.\\

\noindent {\bf Acknowledgments}. The authors thank Yong Wei for helpful discussions.

\section{Preliminaries}
	%=================================================

In this section, we collect some preliminary facts on the geometry of curves in the hyperbolic plane and derive the evolution equations for the relevant geometric quantities along the flow \eqref{1.1}.

	%%%%%%%%%%%%%%%%%%%%%%%%%%%%%%%%%%%%%%%%%%%%%
	\subsection{Hyperbolic geometry}\label{sect:2.1}~
	%%%%%%%%%%%%%%%%%%%%%%%%%%%%%%%%%%%%%%%%%%%%%
	
The hyperbolic space $\mathbb{H}^{n+1}$, $n\geq 1$ can be viewed as a warped product
manifold
$(R_+ \times S^n,\overline{g})$ with
 \begin{equation*}
   \overline{g}=dr^2 +\varphi^2 (r) g_{s^n},
 \end{equation*}
where $\varphi(r) =\sinh r$ and $g_{s^n}$ is the round metric on unit sphere $S^n$.
Let $\overline{\nabla}$ be the Levi-Civita connection on $\mathbb{H}^{n+1}$.
It is well known that the vector field $V=\varphi(r)\partial r$ is a conformal Killing
field satisfying $\overline{\nabla} V=\varphi'(r)\overline{g}$.
\par
Let $M$ be a smooth closed curve  in $H^2$.
The induced metric, the unit inner normal vector and the curvature of $M$ are denoted by $g$, $\nu$ and $\kappa$, respectively.
Here convexity means that the curvature satisfies $\kappa>0$, while h-convexity is equivalent to $\kappa\geq 1$.
\par
Recall that the inner radius $\rho_-$ and outer radius $\rho_+$ of a bounded
domain $\Omega$ are defined by
\begin{equation*}
  \rho_- = sup\{\rho:B_{\rho}(p) \subset  \Omega ~for~ some~ p\in \mathbb{H}^{n+1}\}
\end{equation*}
and
\begin{equation*}
  \rho_+ = inf\{\rho:\Omega  \subset B_{\rho}(p)  ~for~ some~ p\in \mathbb{H}^{n+1}\},
\end{equation*}
where $B_{\rho}(p)$ denotes the geodesic ball of radius $\rho$ centered
at $p$ in $\mathbb{H}^{n+1}$.
If $\Omega$ is $h$-convex in $\mathbb H^2$ (equivalently, if its boundary $\partial\Omega$ is $h$-convex), then its outer radius is uniformly controlled by its inner radius (see \cite{4,14}):
\begin{equation}\label{}
  \rho_+\leq c(\rho_- +\rho_-^{\frac{1}{2}})
  \label{2.1}
\end{equation}
for some constant $c>0$. This estimate is crucial for obtaining a priori $C^0$ estimates for hyperbolic curvature flow, which explains why previous work has focused on the $h$-convexity condition.

For a convex domain $\Omega$ in the hyperbolic plane $\mathbb{H}^2$, \eqref{2.1} is not directly applicable. Instead, the following estimates for $\rho_-$ and $\rho_+$ in terms of the isoperimetric deficit hold (see \cite[\S 18.6]{15}).
\begin{lemma}
Let $\Omega$ be a bounded convex domain in the hyperbolic plane $\mathbb{H}^2$. Set
$L:=\operatorname{Length}(\partial\Omega)$, $A:=\operatorname{Area}(\Omega),
\Delta:=L^2-4\pi A-A^2,$
where $\Delta$ denotes the isoperimetric deficit. Then the following estimates for the inner and outer radii in terms of the isoperimetric deficit hold:
\begin{equation}\label{2.2}
\Delta \geq \bigl[L-A\coth(\tfrac{\rho_-}{2})\bigr]^2,
\end{equation}
\begin{equation}\label{2.3}
\Delta \geq \bigl[A\coth(\tfrac{\rho_+}{2})-L\bigr]^2.
\end{equation}
\end{lemma}
%=================================================
	%%%%%%%%%%%%%%%%%%%%%%%%%%%%%%%%%%%%%%%%%%%%%
	\subsection{Evolution equations along the flow (1.1)}\label{sect:2.2}~
	%%%%%%%%%%%%%%%%%%%%%%%%%%%%%%%%%%%%%%%%%%%%%

Let $X(u,t)$, $u\in S^1$ be the position of the evolving curves
$\gamma_t=X(S^1,t)$ and $g=|\nabla_u X|$.
We define the arc-length parameter $s$ by
\begin{equation*}
  \frac{\partial }{\partial s}=\frac{1}{g}\frac{\partial}{\partial u}.
\end{equation*}
Denote the unit tangent vector by
\begin{equation*}
  \mathbf{t}=\frac{\partial X}{\partial s}=\frac{1}{g}\frac{\partial X}{\partial u}.
\end{equation*}
The arclength element of the curve $\gamma_t$ is represented by $ds=g du$.
The curvature is determined by the derivative of $\mathbf{t}$ with respect to the arc-length parameter through $\partial_s \mathbf{t}=\kappa \nu.$
\par
The following evolution equations are standard and can be derived as in \cite{10}.
\begin{lemma}
Along the flow \eqref{1.1} in the hyperbolic plane $\mathbb{H}^2$,
we have the following evolution equations.\\
(1) The evolution of the length element $ds$ is given by
 \begin{equation}\label{1}
   \frac{\partial}{\partial t}ds=-(\phi_i -\kappa^{\alpha})\kappa ds.
 \end{equation}
(2) The derivative of $\nu$ with respect to time $t$ is
 \begin{equation}
   \nabla_t \nu=\frac{\partial }{\partial s}\kappa^{\alpha} \mathbf{t}.
 \end{equation}
(3) The evolution of curvature $\kappa$ is given by
 \begin{equation}
    \frac{\partial \kappa}{\partial t}=-\frac{\partial^2}{\partial s^2}(\kappa^{\alpha})+(\kappa^2-1)(\phi_i-\kappa^{\alpha}).
    \label{2.6}
 \end{equation}
\end{lemma}
	
	%%%%%%%%%%%%%%%%%%%%%%%%%%%%%%%%%%%%%%%%%%%%%
	\section{Proof of Theorem 1.1}\label{sect:3}
%%%%%%%%%%%%%%%%%%%%%%%%%%%%%%%%%%%%%%%%%%%%%

    \subsection{Preserving convexity }\label{sect:3.1}
	\begin{lemma}\label{lem:3.1}
 If the initial curve $\gamma_0$ is smooth, closed and convex, then the solution $\gamma_t$ of \eqref{1.1} remains convex for  $t>0$.
    \end{lemma}
\begin{proof}
The proof is based on the application of the maximum principle to the evolution equation \eqref{2.6} for the curvature $\kappa$. By assumption, the initial curve $\gamma_0$ is convex, that is, its curvature $\kappa$ is positive on $\gamma_0$. By continuity, $\kappa$ remains positive on $\gamma_t$ for a short time. For each $t_0>0$, let $u_0\in S^1$ be a point such that
\begin{equation*}
  \kappa(u_0,t_0)=\min_{\gamma_{t_0}} \kappa(\cdot,t_0).
\end{equation*}
Then, at $(u_0,t_0)$, one has
$\frac{\partial \kappa}{\partial s}=0$ and
$\frac{\partial^2 \kappa}{\partial s^2}\geq 0.$
By the definitions of $\phi_i(t)$ in \eqref{1.2} and \eqref{1.3}, we have
$\phi_i(t_0)\leq \kappa^\alpha(u_0,t_0).$
If $\kappa(u_0,t_0)\geq 1$, the conclusion follows immediately. Otherwise, \eqref{2.6} yields
\begin{equation*}
  \left.\frac{\partial}{\partial t}\kappa\right|_{(x_0,t_0)} \geq 0.
\end{equation*}
The parabolic maximum principle (see \cite{11}) implies that
  \begin{equation*}
    \kappa(u_0,t_0)\geq min\{\min_{\gamma_{t_0}} \kappa(\cdot,0),1\}=:\underline{\kappa}_0>0.
  \end{equation*}
This proves the lemma.
 \end{proof}
 \subsection{Monotonicity }\label{sect:3.2}~

Let $L(t)$ denote the length of $\gamma_t$ and $A(t)$ denote the area of the
domain $\Omega_t$ enclosed by $\gamma_t$.
Define the isoperimetric deficit by
\begin{equation*}
  \Delta(t)=L^2(t)-4\pi A(t)-A^2(t).
\end{equation*}
We recall the following generalized H\"older inequality (see  [1, Lemma I3.3]):
\begin{lemma}\label{lem:3.2}
Let $M$ be a compact manifold with a volume form $d\mu$,
and let $\xi$ be a continuous function on $M$.
Then for any non-decreasing function $F$,
\begin{equation*}
\int_M  \xi d\mu \int_M F(\xi) \, d\mu \leq \int_M d\mu \int_M \xi F(\xi) \, d\mu.
\end{equation*}
If $F$ is strictly increasing, then equality holds if and only if  $\xi$ is a constant.
\end{lemma}
Then we can prove the following monotonicity of the isoperimetric deficit.
\begin{lemma}\label{lem:3.3}
Along the flow \eqref{1.1}, the isoperimetric deficit $\Delta(t)$ is monotone decreasing.
\end{lemma}
\begin{proof}
With $\phi_1(t)$ given by \eqref{1.2}, we have
\begin{equation}\label{}
  \frac{\partial A}{\partial t}=-\int_{\gamma_t}(\phi_1(t)-\kappa^{\alpha})ds=0,
  \label{3.1}
\end{equation}
\begin{equation}\label{}
  \frac{\partial L}{\partial t}
  \\
  =-\int_{\gamma_t} \kappa(\phi_1(t)-\kappa^{\alpha})ds
  \\
  =\frac{1}{L}(\int_{\gamma_t} gdu \int_{\gamma_t}g\kappa^{\alpha+1}du-\int_{\gamma_t} \kappa gdu \int_{\gamma_t}g\kappa^{\alpha}du).
  \label{3.2}
\end{equation}
Applying the following two H\"older inequalities to equation \eqref{3.2}
\begin{equation}\label{}
  \int_{\gamma_t} gdu\leq(\int_{\gamma_t} g\kappa^{\alpha}du)^{\frac{1}{1-\alpha}} (\int_{\gamma_t} \kappa gdu)^{\frac{-\alpha}{1-\alpha}},
  \label{3.3}
\end{equation}
\begin{equation}\label{}
  \int_{\gamma_t}g\kappa^{\alpha+1}du\leq (\int_{\gamma_t} \kappa gdu)^{\frac{1}{1-\alpha}}(\int_{\gamma_t} g\kappa^{\alpha}du)^{\frac{-\alpha}{1-\alpha}}.
  \label{3.4}
\end{equation}
Substituting \eqref{3.3} and \eqref{3.4} into \eqref{3.2} yields the following equation
\begin{equation}\label{}
  \frac{\partial L}{\partial t}\leq 0.
  \label{3.5}
\end{equation}
Moreover, equality holds in \eqref{3.5} if and only if $\kappa$ is a constant
on $\gamma_t$ which means that $\gamma_t$ is a geodesic circle.
Combining \eqref{3.1} and \eqref{3.5} implies that  $\Delta(t)$  is monotone decreasing.
\par
Similarly, with $\phi_2(t)$ given by \eqref{1.3}, we have
 \begin{equation}\label{}
    \frac{\partial A}{\partial t}=\frac{1}{\int_{\gamma_t} \kappa ds}(\int_{\gamma_t} ds \int_{\gamma_t}\kappa^{\alpha+1}ds-\int_{\gamma_t} \kappa ds \int \kappa^{\alpha}ds)\geq 0,
    \label{3.6}
  \end{equation}
and
  \begin{equation}\label{}
    \frac{\partial L}{\partial t}=\int_{\gamma_t} \kappa(\phi_2(t)-\kappa^{\alpha})ds=0.
  \end{equation}
It follows that the isoperimetric deficit  $\Delta(t)$ is decreasing as well.
\end{proof}
\subsection{Long time existence and convergence }
 \begin{lemma}\label{lem:3.4}
Let $\gamma_t$, $t\in[0,T)$ be a smooth solution of the flow \eqref{1.1}
starting from a smooth, closed and convex curve $\gamma_0$.
Denote by $\rho_-(t)$ and $\rho_+(t)$ the inner and outer radii of the domain
$\Omega_t$ enclosed by $\gamma_t$.
Then there exist positive constants $c_1$, $c_2$ depending only on $\gamma_0$ such that
\begin{equation}\label{3}
   0<c_1\leq \rho_- \leq \rho_+\leq c_2
 \end{equation}
 for all time $t\in[0,T)$.
\end{lemma}
\begin{proof}
Firstly, the lower bound on the inner radius $\rho_-(t)$ follows from the isoperimetric
inequality \eqref{2.2} and the monotonicity in Lemma \ref{lem:3.3}.
We have
\begin{equation}\label{}
  \coth(\frac{\rho_-(t)}{2})\leq\frac{L(t)+\sqrt{\Delta(t)}}{A(t)}
  \leq\frac{L(0)+\sqrt{\Delta(0)}}{A(0)}.
  \label{3.9}
\end{equation}
Note that the right hand side of (3.9) is strictly larger than 1.
Since coth is a strictly decreasing function, we obtain
\begin{equation}\label{}
  \rho_-(t)\geq2\coth^{-1}(\frac{L(0)+\sqrt{\Delta(0)}}{A(0)}).
\end{equation}

Next, to derive an upper bound for $\rho_+$, we cannot apply \eqref{2.3} directly. Indeed, arguing as in (3.9), we obtain only
\begin{equation}\label{}
  \coth(\frac{\rho_+(t)}{2})\geq\frac{L(t)-\sqrt{\Delta(t)}}{A(t)}
  \geq\frac{4\pi+A(0)}{L(0)+\sqrt{\Delta(0)}}.
  \label{3.11}
\end{equation}
However, we do not know whether the right-hand side of \eqref{3.11} is greater than $1$. If it is not, then \eqref{3.11} does not yield an upper bound for $\rho_+(t)$.
Instead, we derive an upper bound for $\rho_+(t)$ by means of the Alexandrov reflection method, following \cite[\S 4]{2}.
The method developed in \cite[\S 4]{2} was used to establish the shape estimate for hypersurfaces with positive sectional curvatures. The same argument also applies to convex curves, and we include it here for the convenience of the reader.
\par
Let $\gamma$ be a geodesic line in $\mathbb{H}^2$,
and let $H_{\gamma}(s)$ be the geodesic line in $\mathbb{H}^2$ which is perpendicular to
$\gamma$ at $\gamma(s)$,$s \in \mathbb{R}$.
We use the notation $H^+_s$ and $H^-_s$ for the half-planes in $\mathbb{H}^2$ determined by
$H_{\gamma}(s)$:
\begin{equation*}
  H_s^+:=\bigcup_{s'\geq s}H_{\gamma(s')},H_s^-:=\bigcup_{s'\leq s}H_{\gamma(s')}.
\end{equation*}
For a bounded domain $\Omega$ in $\mathbb{H}^2$,
denote \(\Omega^+(s) = \Omega \cap H^+_s\) and  $\Omega^-(s) = \Omega \cap H^-_s$.
The reflection map across $H_{\gamma}(s)$ is denoted by $R_{\gamma,s}$.
We define
\begin{equation*}
  S^+_{\gamma}(\Omega) := \inf\{ s \in \mathbb{R} \mid R_{\gamma,s}(\Omega^+(s)) \subset \Omega^-(s) \}.
\end{equation*}

It has been proved in \cite[Lemma 6.1]{3} that for any geodesic line \(\gamma\)
in $\mathbb{H}^2$, $S^+_\gamma(\Omega_t)$ is strictly decreasing along the flow
\eqref{1.1} unless $R_{\gamma,s}(\Omega_t) = \Omega_t$ for some $s \in \mathbb{R}$.
To prove this property in \cite[Lemma 6.1]{3},
the authors rewrite the flow as a uniformly parabolic PDE for the radial function and apply
the strong maximum principle.
Uniform parabolicity there was guaranteed by the $h$-convexity condition together with a uniform upper bound for the curvature.
In our case, we already proved the convexity of the curve $\gamma_t$ in Lemma 3.1.
To deal with the uniform upper curvature bound,
we first fix a time $T_1 \in [0, T)$ and then we have $\underline{\kappa}_0 \leq \kappa
\leq C(T_1)$ on this closed time interval $t \in [0, T_1]$.
Then the argument in \cite[Lemma 6.1]{3} implies that $S^+_\gamma(\Omega_t)$ is s
trictly decreasing in the interval $[0, T_1]$.
Since $T_1 < T$ is arbitrary, this property holds for all time $t \in [0, T)$.
\par
Choose $R>0$ such that the initial domain $\Omega_0$ is contained in some geodesic ball $B_R(p)$ of radius $R$ centered at a point $p\in\mathbb{H}^2$.
We claim that $\Omega_t\cap B_R(p)\neq\emptyset$ for all  $t\in[0,T).$
Indeed, suppose to the contrary that there exists some time $t\in[0,T)$ such that
$\Omega_t\cap B_R(p)=\emptyset.$
Choose a geodesic line $\gamma(s)$ such that there exists a geodesic line $L=H_\gamma(s_0)$ which is perpendicular to $\gamma$ and tangent to the geodesic circle $\partial B_R(p)$,
and such that $\Omega_t$ lies in the half-plane $H_{s_0}^+$. Since $\Omega_0\subset B_R(p)$, we have $R_{\gamma,s_0}\bigl(\Omega_0^+(s_0)\bigr)\subset \Omega_0^-(s_0),$
and hence $S_\gamma^+(\Omega_0)\le s_0$. Since $S_\gamma^+(\Omega_t)$ is decreasing along the flow, it follows that $S_\gamma^+(\Omega_t)\le s_0,$
and therefore
$R_{\gamma,s_0}\bigl(\Omega_t^+(s_0)\bigr)\subset \Omega_t^-(s_0).$
However, since $\Omega_t\subset H_{s_0}^+$, we have $\Omega_t^-(s_0)=\emptyset$, whereas
$R_{\gamma,s_0}\bigl(\Omega_t^+(s_0)\bigr)\neq\emptyset.$
This contradiction proves the claim.
\par
For any $t\in[0,T)$, let $x_1,x_2$ be points on $\gamma_t=\partial\Omega_t$ such that
$d(p,x_1)=\min\{d(p,x):x\in\gamma_t\}, d(p,x_2)=\max\{d(p,x):x\in\gamma_t\},$
where $d(\cdot,\cdot)$ denotes the hyperbolic distance. We first show that $x_1\in B_R(p)$. If not, then $d(p,x)\ge d(p,x_1)>R$ for all  $x\in\gamma_t.$
Since we have already proved that $\Omega_t\cap B_R(p)\neq\emptyset$, convexity implies that
$B_R(p)\subset \Omega_t.$
Hence $\Omega_0\subset B_R(p)\subset \Omega_t.$
Since both $\Omega_0$ and $\Omega_t$ are convex, the monotonicity of the boundary length with respect to the inclusion of convex sets implies that $L(t)>L(0)$ (see [19, \S2]), which contradicts the facts that $L(t)\le L(0)$ in the area-preserving case and $L(t)=L(0)$ in the length-preserving case. Therefore, $x_1\in B_R(p)$ in both cases.
If $x_2 \in B_R(p)$, then the diameter of $\Omega_t$ is bounded from above by $2R$.
Therefore it suffices to study the case $x_2 \notin B_R(p)$.
Let $\gamma(s)$ be the geodesic line passing through $x_1$ and $x_2$, i.e.,
there are $s_1 < s_2 \in \mathbb{R}$ such that $\gamma(s_1) = x_1$ and $\gamma(s_2) = x_2$.
We choose the geodesic line $L =H_{\gamma}(s_0)$ for some $s_0 \in (s_1, s_2)$ such that $L$
is perpendicular to $\gamma$ and is tangent to the boundary of $B_R(p)$ at $p \in \partial
B_R(p)$.
Let $q = \gamma(s_0)$ be the intersection point $\gamma \cap L$.
By the Alexandrov reflection property, $d(x_2, q) \leq d(x_1, q)$.
Then the triangle inequality implies:
\begin{equation}
\begin{split}
\begin{array}{rl}
d(p,x_2) & \leq d(p,x_1) + d(x_1,x_2) \leq d(p,x_1) + 2d(q,x_1) \\
& \leq d(p,x_1) + 2\bigl(d(q,p') + d(p',p) + d(p,x_1)\bigr) \leq 7R,
\end{array}
\end{split}
\end{equation}
where we used the fact that $x_1 \in B_R(p)$.
This shows that the diameter of $\Omega_t$ is uniformly bounded along the flow \eqref{1.1},
and the upper bound of $\rho_+(t)$ follows immediately.
\end{proof}

\begin{lemma}\label{lem:3.5}
Assume that $\alpha<0$ and that $\gamma_0$ is a smooth, closed and convex curve. Let
$\gamma(\cdot,t):S^1\times[0,T)\to\mathbb{H}^2$
be a convex solution of the flow \eqref{1.1} such that its curvature $\kappa(\cdot,t)$ does not blow up as $t\to T$. Then the evolving curve $\gamma(\cdot,t)$ converges to a smooth, closed and convex curve $\gamma(\cdot,T)$ as $t\to T$, and the flow \eqref{1.1} can be continued beyond time $T$.
\end{lemma}

\begin{proof}
The proof is the similar to that of \cite[Lemma 3.2]{7}, and we therefore omit the details.
\end{proof}

\begin{lemma}
Assume that $\alpha<0$ and that $\gamma_0$ is a smooth, closed and convex curve. Consider the flow \eqref{1.1}, and assume that the curvature $\kappa$ of $\gamma(\cdot,t)$ does not blow up in finite time during the evolution. Then the flow admits a unique convex solution
$\gamma(\cdot,t):S^1\times[0,\infty)\to\mathbb H^2$ defined for all time.
\end{lemma}

\begin{proof}
This follows immediately from Lemma \ref{lem:3.1} and Lemma \ref{lem:3.5}.
\end{proof}
\par

The curvature estimates obtained above, together with standard parabolic regularity theory, yield higher-order estimates. We first recall the graphical representation of curves in $\mathbb H^2$. A convex curve $\gamma$ is star-shaped with respect to the center $p$ of an inner ball. Hence it can be written as the graph of a radial function $\rho:S^1\to(0,\infty)$ over $S^1$ in the warped-product model of $\mathbb H^2$. For such a graph,
\begin{equation}
\nu=\frac{1}{\sqrt{1+\frac{\rho_\theta^2}{\sinh^2(\rho)}}}\left(-\partial_r+\frac{\rho_\theta}{\sinh^2(\rho)}\partial_\theta\right),
\end{equation}
\begin{equation}
\kappa=
\frac{\sinh^2(\rho)\cosh(\rho)+2\rho_\theta^2\cosh(\rho)-\rho_{\theta\theta}\sinh(\rho)}
{\left(\sinh^2(\rho)+\rho_\theta^2\right)^{3/2}},
\label{3.9}
\end{equation}
where $\rho_\theta$ and $\rho_{\theta\theta}$ denote the derivatives of $\rho(\theta,t)$ with respect to $\theta\in S^1$, and $\partial_\theta$ and $\partial_r$ denote the angular and radial vector fields, respectively. For each fixed $t_0\in[0,\infty)$, we write $\gamma_{t_0}$ as the graph of a radial function $\rho_{t_0}$ over $S^1$ centered at $p_{t_0}$. By continuity, for $t$ sufficiently close to $t_0$, the curve $\gamma_t$ is still star-shaped with respect to $p_{t_0}$ and hence can be expressed as a radial graph centered at $p_{t_0}$. It is well known that, up to a tangential diffeomorphism, the flow  \eqref{1.1} is equivalent to the following scalar parabolic PDE for the radial function $\rho(\cdot,t)$ on $S^1$ (see \cite{9}):
\begin{equation}
\left\{
\begin{aligned}
&\frac{\partial \rho}{\partial t} = -(\phi_i(t)-\kappa^\alpha)\sqrt{1+\frac{\rho_\theta^2}{\sinh^2(\rho)}}, \qquad t>t_0,\\
&\rho(u,t_0)=\rho_{t_0}(u),
\end{aligned}
\right.
\label{3.10}
\end{equation}
where $\kappa$ is given by \eqref{3.9}. If the maximal existence time of the solution satisfies $T<\infty$, we choose $t_0<T$ sufficiently close to $T$ so that $\gamma_t$ is star-shaped with respect to $p_{t_0}$ for all $t\in[t_0,T)$. The curvature estimates imply uniform $C^2$ estimates for $\rho$. We may then apply the H\"older estimate in one space dimension from \cite[\S7, Chapter XIV]{15} to obtain uniform $C^{2,\alpha}$ estimates for $\rho$ in both space and time. Higher-order regularity follows from parabolic Schauder theory \cite{12}.
\begin{lemma}
Along the area-preserving or length-preserving flow \eqref{1.1}, there exists a family of isometries $\phi_t:\mathbb H^2\to\mathbb H^2$  and a sequence of times $\{t_j\}$, with $t_j\to\infty$, such that $\phi_{t_j}(\gamma_{t_j})$ converges smoothly to a geodesic circle of radius $\rho_\infty$, where $\rho_\infty=\cosh^{-1}(\frac{A(0)}{2\pi}+1)$ if the flow \eqref{1.1} is area-preserving, and $\rho_\infty=\sinh^{-1}(\frac{L(0)}{2\pi})$
if the flow \eqref{1.1} is length-preserving.
\end{lemma}
\begin{proof}
For the area-preserving flow \eqref{1.1} with $\phi_1(t)$  given by \eqref{1.2},
by \eqref{3.5} the length $L(t)$ of $\gamma_t$ satisfies
\begin{equation}\label{}
  \frac{\partial L}{\partial t}=-\int_{\gamma_t} \kappa(\phi_1(t)-\kappa^{\alpha})ds \leq 0.
  \label{3.16}
\end{equation}
The isoperimetric inequality together with preservation of the area implies that
\begin{equation*}
  L(0)\geq L(t)\geq L(\infty) \geq\left(A^{2}(\infty)+4\pi A(\infty)\right)^{1/2}
=\sqrt{A^{2}(0)+4\pi A(0)}.
\end{equation*}
Then the $L^1$ integral of $L^{\prime}(t)$ satisfies
\begin{equation*}
  \int_{0}^{\infty}|L^{\prime}(t)|dt=\int_{0}^{\infty}(-L^{\prime}(t))dt
  =L(0)-L(\infty)<\infty.
\end{equation*}
The uniform bounds of all space-time derivatives of $\kappa(\cdot,t)$ imply that
$L^{\prime}(t)$ is uniformly continuous with respect to $t$.
It follows from the above estimate that $L^{\prime}(t)\to 0$ as $t\to\infty$.
\par
Similarly, for the length-preserving flow \eqref{1.1} with $\phi_2(t)$ defined by \eqref{1.3}, it follows from \eqref{3.6} that the area $A(t)$ of the enclosed domain $\Omega_t$ satisfies
\begin{equation}\label{}
  \frac{\partial A}{\partial t}=\frac{1}{\int_{\gamma_t} \kappa ds}(\int_{\gamma_t} ds \int_{\gamma_t} \kappa^{\alpha+1}ds-\int_{\gamma_t} \kappa ds \int_{\gamma_t} \kappa^{\alpha}ds)\geq0.
  \label{3.17}
\end{equation}
The isoperimetric inequality together with preservation of the length implies that
\begin{equation*}
  A(0)\leq A(t) \leq A(\infty)\leq\sqrt{4\pi^{2}+L(\infty)^{2}}-2\pi
=\sqrt{4\pi^{2}+L(0)^{2}}-2\pi.
\end{equation*}
It follows that the $L^1$ integral of $A^{\prime}(t)$ is finite and $A^{\prime}(t)\to 0$
as $t\to\infty$.
\par
Let $p_t$ be the center of an inner ball of $\Omega_t$, and let
$\varphi_t:\mathbb{H}^2\to\mathbb{H}^2$ be an isometry carrying $p_t$ to $p_0$.
Clearly, each $\tilde{\gamma}_t=\varphi_t(\gamma_t)$ is a convex curve with a center
of an inner ball at the fixed point $p_0$ and inner radius $\rho_{-}(t)\geq c_1$.
The regularity estimate on $\gamma_t$ carries directly to the regularity estimate
of $\tilde{\gamma}_t$.
This implies that there exists a sequence of times $\{t_j\},t_j\to\infty$,
such that $\tilde{\gamma}_{t_j}$ converges smoothly to a limit $\tilde{\gamma}_{\infty}$.
Since the area, length and curvature of $\gamma_t$ are invariant under $\varphi_t$,
using $L^{\prime}(t)\to 0$ or $A^{\prime}(t)\to 0$ as $t\to\infty$,
we have
\begin{equation}\label{}
  L(\infty)\int_{\widetilde{\gamma}_{\infty}} \kappa^{\alpha+1}ds-\int_{\widetilde{\gamma}_{\infty}} \kappa ds \int_{\widetilde{\gamma}_{\infty}} \kappa^{\alpha}ds=0
  \label{3.18}
\end{equation}
by \eqref{3.16} and \eqref{3.17}.
Applying Lemma \ref{lem:3.2} to \eqref{3.18} with $\xi=\kappa$ and $F(\xi)=\xi^\alpha$, we conclude that $\tilde{\gamma}_\infty$ has constant curvature and hence $\tilde{\gamma}_\infty$ is a geodesic circle of some radius $\rho_\infty$.
The radius $\rho_{\infty}$ can be determined uniquely by the initial curve $\gamma_0$.
If the global term $\phi_1(t)$ is chosen  to preserve the enclosed area $A(t)$,
we have
\begin{equation*}
  A(0)=A(\infty)=2\pi\int_{0}^{\rho_{\infty}}\sinh rdr=2\pi\left(\cosh \rho_{\infty}-1\right).
\end{equation*}
This implies that $\rho_{\infty}=\cosh^{-1}(\frac{A(0)}{2\pi}+1)$.
Similarly, if the flow is length preserving, we have
\begin{equation*}
  L(0)=L(\infty)=2\pi\sinh\rho_{\infty},
\end{equation*}
from which we have $\rho_{\infty}=\sinh^{-1}(\frac{L(0)}{2\pi})$.
\end{proof}

The exponential convergence is obtained by studying the linearization of the flow \eqref{1.1}. The argument is similar to that in Section 6 of \cite{5}.
For each sufficiently large time $t_k$, we write $\gamma_{t_k}$ as the graph of a radial function $\rho_{t_k}$ over $S^1$, centered at $p_{t_k}$.
For $t$ sufficiently close to $t_k$, we rewrite the flow \eqref{1.1} as the following scalar parabolic equation:
\begin{equation}
\left\{
\begin{aligned}
&\frac{\partial \rho}{\partial t} = -( \phi_i(t) - \kappa^{\alpha}) \sqrt{1+\frac{\rho_{\theta}^2}{\sinh^2(\rho)}},  ~~~~t>t_k\\
&\rho(u,t_k) = \rho_{t_k}(u),
\end{aligned}
\right.
\label{3.19}
\end{equation}
where $\kappa$ is given in terms of $\rho$ by \eqref{3.9}.
Note that we can assume the oscillation of $\rho_{t_k}-\rho_{\infty}$ is sufficiently small
by choosing $t_k$ large enough.
\begin{lemma}
The linearization of \eqref{3.19} about the geodesic circle of radius $\rho_{\infty}$ is
\begin{equation}\label{}
  \frac{\partial \eta}{\partial t}=-\frac{\alpha \tanh^{-\alpha-1}\rho_{\infty}}{\cosh^2\rho_{\infty}}
  (\eta_{\theta\theta}+\eta-\frac{1}{2\pi}\int_0^{2\pi}\eta d\theta).
  \label{3.20}
\end{equation}
\end{lemma}
\begin{proof}
We write $\rho=\rho_{\infty}(1+\varepsilon\eta)$ for $\varepsilon$ sufficiently small. A direct computation yields
\begin{equation}\label{}
  \frac{\partial}{\partial \varepsilon}\bigg{|}_{\varepsilon=0}\kappa
=-\frac{\rho_{\infty}}{\sinh^{2}\left(\rho_{\infty}\right)}
(\eta_{\theta\theta}+\eta),
\label{3.21}
\end{equation}
\begin{equation}\label{}
  \frac{\partial}{\partial \varepsilon}\bigg{|}_{\varepsilon=0}
\sqrt{1+\rho_{\theta}^{2}/\sinh^{2}(\rho)}=0,
\label{3.22}
\end{equation}
\begin{equation}\label{}
  \frac{\partial}{\partial \varepsilon}\bigg{|}_{\varepsilon=0}
\sqrt{\rho_{\theta}^{2}+\sinh^{2}(\rho)}=\rho_{\infty}\cosh(\rho_{\infty})\eta.
\label{3.23}
\end{equation}
Note that, in the graphical representation, the global term $\phi_1(t)$ defined by \eqref{1.2} can be rewritten as
\begin{equation*}
  \phi_1(t)=\frac{\int_{\mathbb{S}^{1}}
\kappa^{\alpha}\sqrt{\rho_{\theta}^{2}+\sinh^{2}\left(\rho\right)}d\theta}
{\int_{\mathbb{S}^{1}}\sqrt{\rho_{\theta}^{2}+\sinh^{2}\left(\rho\right)}d\theta}.
\end{equation*}
We compute that
\begin{equation}\label{}
  \frac{d}{d\varepsilon}\bigg{|}_{\varepsilon=0}\phi_1
=-\frac{\alpha\rho_{\infty}\tanh^{-\alpha-1}(\rho_{\infty})}
{2\pi\cosh^{2}\left(\rho_{\infty}\right)}\int\limits_{\mathbb{S}^{1}}\eta d\theta.
\label{3.24}
\end{equation}
Similarly, \eqref{3.24} also holds with $\phi_2(t)$ in place of $\phi_1(t)$. Combining \eqref{3.21}--\eqref{3.24}, we obtain \eqref{3.20}. This completes the proof.
\end{proof}

Since the oscillation of $\rho_{t_k}-\rho_\infty$ is sufficiently small, the stability argument in \cite{5}, together with the result of \cite{6}, implies that the solution of \eqref{3.19} with initial data $\rho_{t_k}$ exists for all time and converges exponentially to the constant $\rho_\infty$.
This means that the curve $\overline{\gamma_t}=$ graph $\rho(\cdot,t)$ solves \eqref{1.1}
with initial condition $\gamma_{t_k}$ and $\overline{\gamma_t}$ converges exponentially to
a geodesic circle of radius $\rho_{\infty}$.
By uniqueness of solutions, the solution constructed from the initial curve $\gamma_{t_k}$ coincides with the original solution for  $t\geq t_k$, and hence the solution of \eqref{1.1} with initial condition $\gamma_0$ converges exponentially to the geodesic circle of radius $\rho_\infty$ without applying an ambient isometry.
This completes the proof of Theorem 1.1.
	%%%%%%%%%%%%%%%%%%%%%%%%%%%%%%%%%%%%%%%%%%%%%
	\section{A Sufficient Condition for Global Existence}\label{sect:4}
	%%%%%%%%%%%%%%%%%%%%%%%%%%%%%%%%%%%%%%%%%%%%%
In this section, we establish a sufficient condition on the initial curve $\gamma_0$ that guarantees global existence of the convex solution to \eqref{1.1}.
\begin{lemma}\label{lem:4.1}
Let  $V=\frac{1}{\kappa},$ and define
$\eta(t):=\gamma_t\cap\{\phi_i(t)\ge V^{-\alpha}\}\cap\left\{V^2<\frac32\right\}.$
Assume that for  $t\in[0,\infty)$ there exists $s(t)\in S^1$ such that
$V(s(t),t)=D,$ where
\begin{equation*}
   D=\cosh^{-1}(\frac{A(0)}{2\pi}+1)
~or~
D=\sinh^{-1}(\frac{L(0)}{2\pi}),
 \end{equation*}
according to whether the flow is area-preserving or length-preserving. Then there exists a positive constant $C$, depending only on $\alpha$ and $\gamma_0$, such that
\begin{equation}
0 < \kappa \leq C, \quad \forall (s, t) \in S^1 \times [0, \infty),
\label{4.1}
\end{equation}
or equivalently
\begin{equation}
V \geq C^{-1}, \quad \forall (s, t) \in S^1 \times [0, \infty).
\label{4.2}
\end{equation}
\end{lemma}

 \begin{proof}
By the evolution equation \eqref{2.6} for the curvature $\kappa$, the function $V$ satisfies
\begin{equation*}
\frac{\partial}{\partial t } V
=
V^2 \frac{\partial^2}{\partial s^2} V^{-\alpha}
-(1-V^2)(\phi_i-V^{-\alpha}).
\end{equation*}
Firstly, we consider the case in which $\phi_i(t)<V^{-\alpha}=\kappa^\alpha$.
\par
Define $W:=\kappa+\frac{1}{\kappa}.$
Then, by a direct computation, one obtains
  \begin{equation*}
    \frac{\partial}{\partial s} W=(1-\kappa^{-2})\frac{\partial}{\partial s}\kappa,
  \end{equation*}
  and
  \begin{equation*}
    \frac{\partial^2}{\partial s^2} W=(1-\kappa^{-2})\frac{\partial^2}{\partial s^2}\kappa+2\kappa^{-3}(\frac{\partial}{\partial s} \kappa)^2.
  \end{equation*}
Combining the above identities, we obtain the following evolution equation for $W$.
 \begin{equation}
   \begin{split}
  \frac{\partial }{\partial t}W
 =&\  (1-\kappa^{-2})\frac{\partial }{\partial t}\kappa
 \\
 =&\  -\alpha \kappa^{\alpha-1}(1-\kappa^{-2})\frac{\partial^2}{\partial s^2} \kappa
 +\alpha(-\alpha+1)\kappa^{\alpha-2}(1-\kappa^{-2})(\frac{\partial}{\partial s} \kappa)^2
 \\
 &\
 +(\kappa^2-1)(1-\kappa^{-2})(\phi_i-\kappa^{\alpha})
 \\
 =&\  -\alpha \kappa^{\alpha-1}\frac{\partial^2}{\partial s^2} W+\alpha(-\alpha+1)\kappa^{\alpha-2} \frac{\partial}{\partial s}\kappa \frac{\partial}{\partial s} W+2\alpha \kappa^{\alpha-4}(\frac{\partial}{\partial s}\kappa)^2
 \\
 &\
 +\frac{(\kappa^2-1)^2}{\kappa^2}(\phi_i-\kappa^{\alpha}).
 \label{4.3}
 \end{split}
 \end{equation}\par
Applying the maximum principle to \eqref{4.3}, we conclude that $W_{\max}(t)$ is non-increasing in time. Therefore,
\begin{equation*}
  \kappa_{\max}(t)<\kappa_{\max}(t)+\frac{1}{\kappa_{\max}(t)}\leq W_{\max}(t)
\leq W_{\max}(0)=\max_{\gamma_0}\left(\kappa+\frac{1}{\kappa}\right).
\end{equation*}
\par
Secondly, consider the case $\phi_i (t)\geq V^{-\alpha}$.
The condition $V^{2}\geq \frac{3}{2}$ yields a lower bound for $V$;
thus, we only need to analyze the case $V^{2}< \frac{3}{2}$.
\begin{equation}
   \begin{split}
  \frac{\partial}{\partial t}\int_{\eta(t)}V^2 ds
 =&\  2\int_{\eta(t)}V V_t ds + \int_{\eta(t)}V^2 (ds)_t
 \\
 =&\  2\int_{\eta(t)}V^3 (V^{-\alpha})_{ss}ds + 2\int_{\eta(t)}V^3 (\phi_i -V^{-\alpha})ds
 - 3\int_{\eta(t)}V(\phi_i -V^{-\alpha})ds
 \\
 =&\ 6\alpha \int_{\eta(t)}V^{-\alpha +1}V_s^2 ds+ 2\int_{\eta(t)}V(V^2 -\frac{3}{2})(\phi_i -V^{-\alpha})ds.
 \label{4.4}
 \end{split}
 \end{equation}
From Equation \eqref{4.4} we have
\begin{equation}\label{}
  \frac{\partial}{\partial t}\int_{\eta(t)}V^2 ds\leq 6\alpha \int_{\eta(t)}V^{-\alpha +1}V_s^2 ds,
\end{equation}
\begin{equation}
   \begin{split}
  -\alpha \int_0^t\int_{\eta(t)}V^{-\alpha +1}V_s^2 ds\leq
&\  -\frac{1}{6}\int_{\eta(t)}V^2 (t)ds + \frac{1}{6}\int_{\eta(t)} V^2 (0)ds
 \\
 \leq&\  \frac{1}{6}\int_{\eta(t)}V^2 (0)ds.
 \end{split}
 \end{equation}
Letting
\begin{equation*}
  f(t) = \int_{\eta(t)}V^{-\alpha +1}V_{s}^{2}ds,\ \ \ \ b = \left(\frac{8}{\underline{\kappa}_0^{2}}\right)^{- 1}D,
\end{equation*}
since $\displaystyle\int_0^\infty f(t)dt$ converges, for any given $\varepsilon >0$,
there are only finitely many $n\in \mathbb{N}$ such that $f(t)\geq \varepsilon$ holds
for  $t\in [nb,(n + 1)b]$.
Then there exists a natural number $N_0 > 0$ such that as long as $n\geq N_0$,
$f$ is less than $\varepsilon$ at some point $t_n\in [nb,(n + 1)b]$, i.e.,
\begin{equation}\label{}
  0\leq f(t_n)< \varepsilon, \quad t_n\in [nb,(n + 1)b],\quad n\geq N_0,
\end{equation}
where $\varepsilon >0$ is a number to be chosen suitably later.
Then, under the assumption that $\eta(t)$ we obtain $V(s_{n},t_{n}) = D$,
\begin{equation}
   \begin{split}
  V_{\min}^{\frac{-\alpha + 1}{2}}(t_n) - D^{\frac{-\alpha + 1}{2}}
 =&\ V^{\frac{-\alpha + 1}{2}}(s_*,t_n) - V^{\frac{-\alpha + 1}{2}}(s_n,t_n)
 \\
 =&\  \int_{s_n}^{s_*} \frac{-\alpha + 1}{2} V^{\frac{-\alpha + 1}{2}}(s,t_n)V_s(s,t_n)ds
 \\
\geq &\ -\frac{-\alpha + 1}{2}\sqrt{L}\sqrt{f(t_n)},
\label{4.8}
 \end{split}
 \end{equation}
where we have used the H\"older inequality in \eqref{4.8}. The above inequality implies
\begin{equation}
   \begin{split}
  V_{\min}^{\frac{-\alpha + 1}{2}}(t_n)\geq D^{\frac{-\alpha + 1}{2}} - \frac{-\alpha + 1}{2}\sqrt{L}\sqrt{f(t_n)}\geq D^{\frac{-\alpha + 1}{2}} - \frac{-\alpha + 1}{2}\sqrt{L(0) \varepsilon}
 \end{split}
 \end{equation}
and if we choose $\varepsilon >0$ small enough (the choice depends only on $\alpha$
and $X_0$ ), we get
\begin{equation}\label{}
  V_{\min}^{\frac{-\alpha + 1}{2}}(t_n)\geq \frac{1}{2} D^{\frac{-\alpha + 1}{2}},\quad t_n\in [nb,(n + 1)b],\quad n\geq N_0.
  \label{4.10}
\end{equation}
\par
Now we focus on the time interval $t\in [(N_0 + 1)b,\infty)$ and look at the minimum
function $V_{\min}(t)$ on $[(N_0 + 1)b,\infty)$.
Firstly, for  $t\in[(N_0+1)b,\infty)$, one can find an integer $n\geq N_0$ and a point $t_n\in[nb,(n+1)b]$ such that $0<t-t_n\leq 2b$ and \eqref{4.10} holds for $V_{\min}(t_n)$.
The idea is to compare $V_{\min}(t)$ and $V_{\min}(t_n)$ for $t\in [(N_0 + 1)b,\infty)$.
Let $s(t)\in S^1$ be such that $V(s(t),t) = V_{\min}(t)$, $t\in [0,\infty)$.
By the evolution equation of $V(s,t)$, at the minimum point $s(t)$,
we have
\begin{equation}\label{}
  \frac{\partial V}{\partial t}(t)\geq -(\phi_i -V^{-\alpha}) + V^{2}(\phi_i -V^{-\alpha})\geq -\underline{\kappa}_0^{\alpha},~~\forall~ t\in [0,\infty).
\end{equation}
Hence the maximum principle implies that the Lipschitz continuous function
$V_{\min}(t) = V(s(t),t)$,
on the time interval $[t_n,t]$, satisfies the estimate
\begin{equation}\label{}
  V_{\min}(t) - V_{\min}(t_n)\geq -\underline{\kappa}_0^{\alpha}(t - t_n)\geq -2b\underline{\kappa}_0^{\alpha},
\end{equation}
where now $t\in [(N_0 + 1)b,\infty)$ and $t_n\in [nb,(n + 1)b]$, $n\geq N_0$.
Since $V_{\min}(t_n)$ satisfies \eqref{4.10}, the above inequality implies
\begin{equation}\label{}
  V_{\min}(t) \ge \frac{1}{2}D - 2b\underline{\kappa}_0^{\alpha}, \quad \forall t \in [(N_0+1)b, \infty)
\end{equation}
and by the choice of $b=\left(\frac{8}{\underline{\kappa}_0^{2}}\right)^{-1}D$, we deduce the lower bound estimate.
\begin{equation}\label{}
  V_{\min}(t)\geq \frac{1}{4} D,\quad \forall~ t\in [(N_0 + 1)b,\infty).
\end{equation}
The continuity of $V$ implies that there exists a positive constant $C$, depending only on $\alpha$ and $\gamma_0$, such that $V_{\min}(t)\geq C$ for  $t\in[0,\infty).$
This proves \eqref{4.2}, and hence also \eqref{4.1}. The proof is complete.
\end{proof}
	\section{Example}
In this section, we present an explicit example for which the curvature remains uniformly bounded for all time along the flow.
\begin{example}
Let $\gamma_0$ be a convex curve with radial function
$\rho_0(\theta)=2+\cos\theta,  \theta\in[0,2\pi].$
Then $1\leq \rho_0(\theta)\leq 3.$
It follows from \eqref{3.9} that the curvature $\kappa$ satisfies
 \begin{equation*}
   \kappa= \frac{\sinh^2(\rho) \cosh(\rho) + 2\sin^2(\theta) \cosh(\rho) +\cos(\theta)\sinh(\rho)}{\left( \sinh^2(\rho) + \sin^2(\theta) \right)^{\frac{3}{2}}}.
 \end{equation*}
A direct calculation shows that
\begin{equation*}
0<\kappa\leq \frac{\sinh^2(\rho)\cosh(\rho)+2\cosh(\rho)+(\rho-2)\sinh(\rho)}
{\bigl[\cosh^2(\rho)-(\rho-2)^2\bigr]^{3/2}}.
\end{equation*}
Since $\rho$ is uniformly bounded, both $\sinh(\rho)$ and $\cosh(\rho)$ remain uniformly bounded. Hence $\kappa$ is uniformly bounded from above. Therefore, this example satisfies the bounded-curvature assumption appearing in Theorem~1.1. A schematic illustration is shown in Figure~1.
\end{example}
	
\begin{figure}[htbp]
    \centering
    \begin{subfigure}[b]{0.8\textwidth}   % 设置子图宽度（例如页面宽度的80%）
        \centering
        \includegraphics[width=\textwidth]{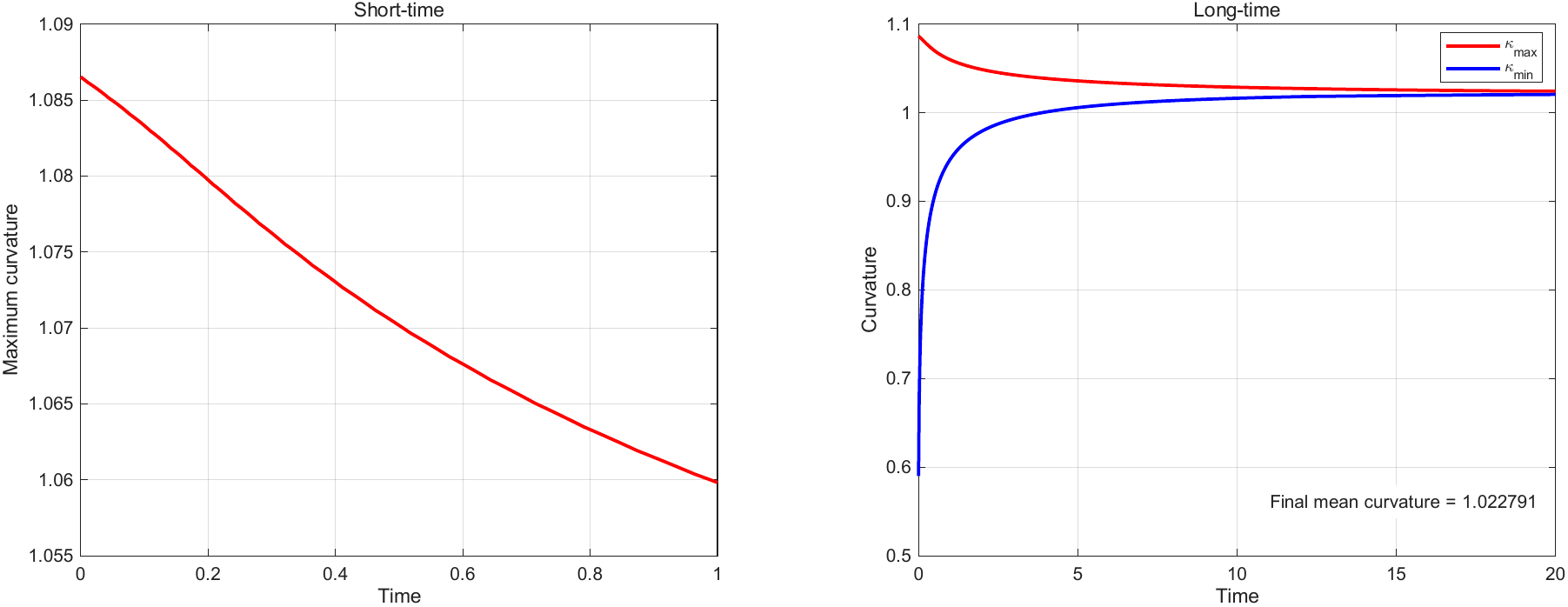}   % 第一张图片
        \caption{area-preserving for $\alpha$=-1}
        \label{fig:top}
    \end{subfigure}

    \vspace{1em}  % 上下两张图之间的垂直间距，可调整

    \begin{subfigure}[b]{0.8\textwidth}
        \centering
        \includegraphics[width=\textwidth]{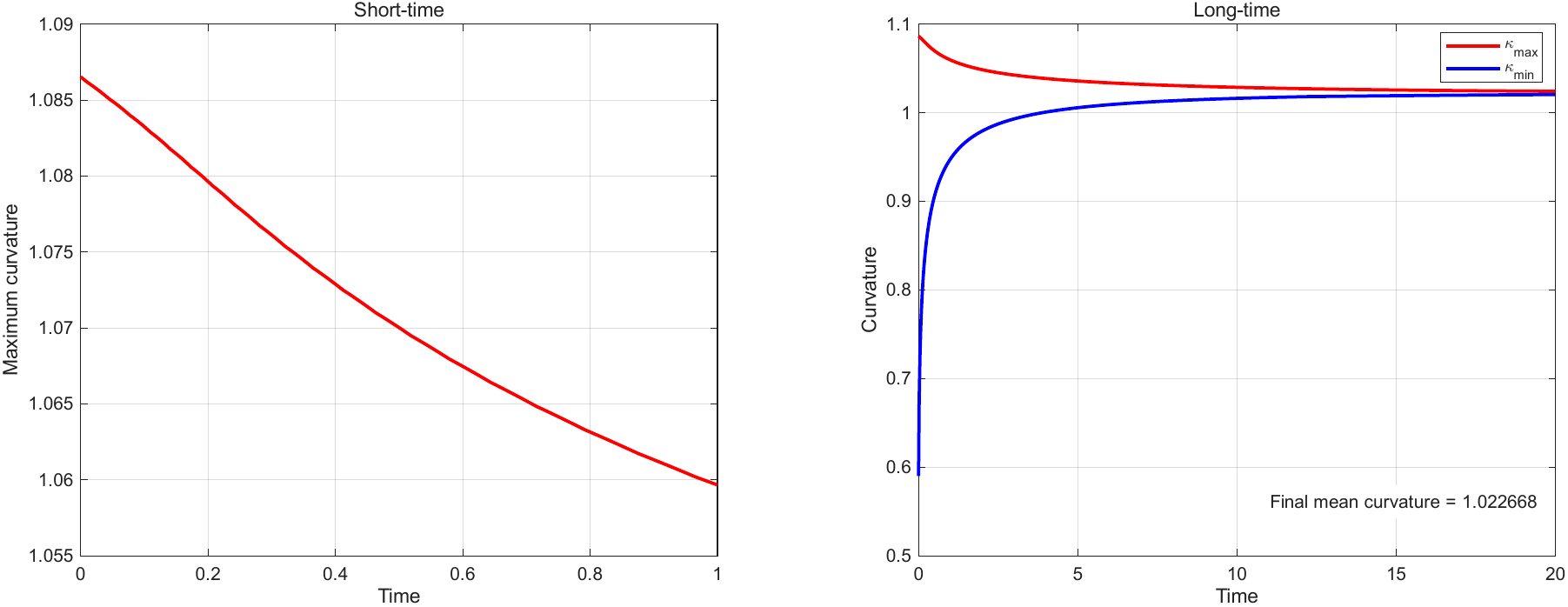}   % 第二张图片（可替换为其他图片）
        \caption{length-preserving for $\alpha$=-1}
        \label{fig:bottom}
    \end{subfigure}

   % \caption{总标题：两张竖排图片}
    \label{fig:vertical}
\end{figure}

	%%%%%%%%%%%%%%%%%%%%%%%%%%%%%%%%%%%%%%%%%%%%%%%%%%%%%%%%%%%%%%%%%%%%

\end{document}